\documentclass{article}
\usepackage[english]{babel}
\usepackage{amsmath,amssymb,graphicx,hyperref,latexsym,theorem}

\catcode`\>=\active \def>{
\fontencoding{T1}\selectfont\symbol{62}\fontencoding{\encodingdefault}}
\newcommand{\assign}{:=}
\newcommand{\backassign}{=:}
\newcommand{\dueto}[1]{\textup{\textbf{(#1) }}}
\newcommand{\infixand}{\text{ and }}
\newcommand{\longdownarrow}{{\mbox{\rotatebox[origin=c]{-90}{$\longrightarrow$}}}}
\newcommand{\longuparrow}{{\mbox{\rotatebox[origin=c]{90}{$\longrightarrow$}}}}
\newcommand{\mathD}{\mathrm{D}}
\newcommand{\mathd}{\mathrm{d}}
\newcommand{\nin}{\not\in}
\newcommand{\nobracket}{}
\newcommand{\tmmathbf}[1]{\ensuremath{\boldsymbol{#1}}}
\newcommand{\tmname}[1]{\textsc{#1}}
\newcommand{\tmnote}[1]{\thanks{\textit{Note:} #1}}
\newcommand{\tmop}[1]{\ensuremath{\operatorname{#1}}}
\newcommand{\tmtextit}[1]{\text{{\itshape{#1}}}}
\newenvironment{itemizedot}{\begin{itemize} }{\end{itemize}}
\newenvironment{proof}{\noindent\textbf{Proof\ }}{\hspace*{\fill}$\Box$\medskip}
\newenvironment{proof*}[1]{\noindent\textbf{#1\ }}{\hspace*{\fill}$\Box$\medskip}
\newtheorem{definition}{Definition}
\newtheorem{lemma}{Lemma}
{\theorembodyfont{\rmfamily}\newtheorem{remark}{Remark}}
\newtheorem{theorem}{Theorem}

\begin{document}

\title{The dyadic and the continuous Hilbert transforms with values in Banach
spaces. Part 2.}

\author{
  Komla Domelevo
  \and
  Stefanie Petermichl
  \tmnote{The second author is supported by the European Research Council
  project CHRiSHarMa no. DLV-682402 and by the Alexander von Humboldt Foundation}
}

\maketitle

\begin{abstract}
  We show that if the dyadic Hilbert transform with values in a Banach space
  is $L^p$ bounded, then so is the Hilbert transform, with a linear relation
  of the bounds. This result is the counterpart of {\cite{DomPet2022a}} where
  the opposite bound was proven. 
\end{abstract}

\section{Introduction}

Let $X$ be a Banach space. Let $\mathbb{D}$ the unit disc and $\partial
\mathbb{D}$ the unit circle. Let $f \in L^p (\partial \mathbb{D}, X)$ be a
$L^p$ integrable $X$--valued function, that is
\[ \| f \|_{L^p (\partial \mathbb{D}, X)} \assign \left( \int_{\partial
   \mathbb{D}} \| f \|_X^p \mathd x \right)^{1 / p} < \infty . \]
We also write in short $L^p (\partial \mathbb{D})$ instead of $L^p (\partial
\mathbb{D}, X)$. It is well known that the Hilbert transfrom $\mathcal{H} :
L^p (\partial \mathbb{D}) \rightarrow L^p (\partial \mathbb{D})$ is bounded if
and only if $X$ is UMD. We wish to compare this operator with the so--called
dyadic Hilbert transform. Let $\Omega \assign [0, 1)$, $\mathcal{D}$ the set
of dyadic intervals, and
\[ f (x) = \langle f \rangle_{I_0} + \sum_{I \in \mathcal{D}} (f, h_I) h_I \]
the Haar decomposition of the $X$--valued function $f$. The dyadic Hilbert
transform $\mathcal{S}$ is the operator sending
\[ \langle f \rangle_{I_0} \mapsto 0, \quad h_{I_0} \mapsto 0, \quad
   h_{I_{\pm}} \mapsto \pm h_{\mp} . \]
It is well known that $\mathcal{S} : L^p (\Omega) \rightarrow L^p (\Omega)$ is
bounded in the UMD Banach space $X$, and moreover we know from the recent
preprint {\cite{DomPet2022a}} that its operator norm is bounded by the
operator norm of $\mathcal{H} : L^p (\partial \mathbb{D}) \rightarrow L^p
(\partial \mathbb{D})$ with a linear dependence. In this part of our study, we
prove the converse.

\begin{theorem}
  \label{T: H less than S}Let $\mathcal{H} : L^p (\partial \mathbb{D})
  \rightarrow L^p (\partial \mathbb{D})$ the Hilbert transform on the disc
  with values in $X$, and $\mathcal{S} : L^p ([0, 1)) \rightarrow L^p ([0,
  1))$ the dyadic Hilbert transfrom with values in $X$. We have
  \[ \| \mathcal{H} \|_{p \rightarrow p} \leqslant \| \mathcal{S} \|_{p
     \rightarrow p} . \]
\end{theorem}

The connection established by the second author in {\cite{Pet2000b}} between
the Hilbert transform $\mathcal{H}_{\mathbb{R}}$ on the real line and the
classical Haar shift $\mathcal{S}_{\tmop{cl}} : h_I \rightarrow (h_{I_+} -
h_{I_-}) / \sqrt{2}$ states that the Hilbert transform is, up to a universal
multiplicative constant, an average of translated and dilated classical Haar
shifts $\mathcal{S}_{\tmop{cl}}^{\alpha, r}$of the form
\[ \mathbb{E}^{\alpha} \mathbb{E}^r  \mathcal{S}_{\tmop{cl}}^{\alpha, r} =
   \mathcal{H}_{\mathbb{R}} . \]
Here $\alpha$ denotes the translation parameter, $r$ the dilation parameter,
$\mathbb{E}^{\alpha}$ the averaging operator with respect to translations, and
$\mathbb{E}^r$ the averaging operator with respect to dilations. Since $\|
\mathcal{S}_{\tmop{cl}}^{\alpha, r} \|_{p \rightarrow p} = \|
\mathcal{S}_{\tmop{cl}} \|_{p \rightarrow p}$, the upper--bound $\|
\mathcal{H}_{\mathbb{R}} \|_{p \rightarrow p} \leqslant \|
\mathcal{S}_{\tmop{cl}} \|_{p \rightarrow p}$ follows. A natural question is
to ask if a similar strategy allows one to prove Theorem \ref{T: H less than
S}. Unfortunately, we prove in Section \ref{S: averaging dyadic Hilbert
transform} that this is not the case since this average procedure yields for
the dyadic Hilbert transform $\mathcal{S}$
\[ \mathbb{E}^{\alpha} \mathbb{E}^r  \mathcal{S}^{\alpha, r} = 0 \]

However, the comparison of the two operators can be achieved through
stochastic representations of the Hilbert transform. It will reveal a profound
connection between $\mathcal{H}$ and $\mathcal{S}$, which was an important
motivation to us for bringing this dyadic Hilbert transform into play.

\

\paragraph{Stochastic representation of the Hilbert transform}Let $f \in L^p
(\partial \mathbb{D})$ given and (with the same notation) $f \in L^p
(\mathbb{D})$ its harmonic extension on the unit disc $\mathbb{D}$. Let
further $g = \mathcal{H} f$ the Hilbert transform of $f$, and $g$ its harmonic
extension. Introduce $W$ the standard two--dimensional Brownian motion started
at the origin. The It{\^o} formula ensures that almost surely,
\[ f (W_t) = f (0, 0) + \int_0^t \nabla f (W_s) \cdot \mathd W_s + \frac{1}{2}
   \int_0^t \Delta f (W_s) \mathd s = f (0, 0) + \int_0^t \nabla f (W_s) \cdot
   \mathd W_s, \]
since $f$ is harmonic. The formula above is valid for all times $t$ such that
$(W_s)_{0 \leqslant s \leqslant t}$ remains in the unit disc. Let $\tau$ the
stopping time
\[ \tau \assign \inf \{ t > 0 ; W_t \nin \mathbb{D} \} . \]
This is the random variable equal to the first instant when the continuous
random walk hits $\partial \mathbb{D}$. We have again thanks to It{\^o}'s
formula,
\[ f (W_{\tau}) = f (0, 0) + \int_0^{\tau} \nabla f (W_s) \cdot \mathd W_s .
\]
In particular if $W_{\tau} = z$, i.e. the random walk hits $\partial
\mathbb{D}$ at $z$, then
\[ f (z) = f (0, 0) + \int_0^{\tau} \nabla f (W_s) \cdot \mathd W_s . \]
Similarly, owing to the Cauchy--Riemann relations for the conjugate function
$g$ of $f$, we have
\[ \forall t \leqslant \tau, \quad g (W_t) = \int_0^t \nabla g (W_s) \cdot
   \mathd W_s = \int_0^t \nabla f (W_s)^{\perp} \cdot \mathd W_s, \]
where
\[ \nabla^{\perp} f (W_s) = \left(\begin{array}{cc}
     0 & - 1\\
     1 & 0
   \end{array}\right) \nabla f (W_s) \]
denotes the vector clockwise orthogonal to $\nabla f (W_s)$. In other words,
$g (W_t)$ is a martingale transform of $f (W_t)$ with predictable multiplier
the rotation matrix above. Finally if $W_{\tau} = z \in \partial \mathbb{D}$,
then
\[ g (z) = \int_0^{\tau} \nabla g (W_s) \cdot \mathd W_s = \int_0^{\tau}
   \nabla f (W_s)^{\perp} \cdot \mathd W_s, \]
In the next sections, we will be dealing with the martingale $\mathcal{M}^f_t
\assign f (W_t)$ and its martingale transform $\mathcal{M}_t^g \assign g
(W_t)$ defined as
\begin{equation}
  \forall t \leqslant \tau, \quad \mathcal{M}_t^f \assign f (0, 0) + \int_0^t
  \nabla f (W_s) \cdot \mathd W_s, \quad \mathcal{M}_t^g \assign \int_0^t
  \nabla^{\perp} f (W_s) \cdot \mathd W_s . \label{eq:Mft Mgt}
\end{equation}
In the case where $f$ and therefore $g$ are Hilbert--valued, the martingales
above are on the one hand differentially subordinate, i.e.
\[ \mathd [\mathcal{M}^f, \mathcal{M}^f]_t = \| \nabla f (W_s) \|^2 \mathd t =
   \| \nabla f (W_s)^{\perp} \|^2 \mathd t = \mathd [M^g, M^g]_t, \]
and on the other hand are orthogonal, i.e
\[ \mathd [\mathcal{M}^f, \mathcal{M}^g]_t = (\nabla f (W_s), \nabla^{\perp} f
   (W_s)) = 0. \]
Those were studied in the work of {\tmname{Ba{\~n}uelos}} and {\tmname{Wang}}
{\cite{BanWan1996}} to which we refer for more details. They prove the sharp
martingale $L^p$ bound $\| \mathcal{M}^g \|_p \leqslant c_p \| \mathcal{M}^f
\|_p$ using special functions, and where $c_p \assign \|
\mathcal{H}_{\mathbb{R}} \|_{p \rightarrow p}$.

In this study, we do not aim at providing directly an absolute constant for
the $L^p$ norm of the Hilbert transform, but rather look for a direct
connection between the Hilbert and dyadic Hilbert transforms that holds also
in the Banach valued case. We establish a tight connection through the use of
stochastic integrals.

\

\paragraph{Strategy and motivation}Due to the discrete nature of the dyadic
shift, we first aim at approximating the two stochastic integrals in
\eqref{eq:Mft Mgt} by using two--dimensional discrete random walks $(B_k)_{k
\geqslant 0}$ built on top of the dyadic system. We expect discrete
martingales $M_k^f$ and $M_k^g$ that should be discrete counterparts of
\eqref{eq:Mft Mgt}, namely
\begin{equation}
  M_k^f (x) \assign f (0, 0) + \sum_{l = 1}^k \nabla f (B_{l - 1} (x)) \cdot
  \mathd B_l (x), \quad M_k^g (x) \assign \sum_{l = 1}^k \nabla^{\perp} f
  (B_{l - 1} (x)) \cdot \mathd B_l (x), \label{eq:Mfn Mgn}
\end{equation}
where $x$ spans the dyadic probability space $\Omega \assign [0, 1)$ with
uniform probability density $\mathd \mathbb{P} (x) = \mathd x$. In order to
relate the dyadic operator $\mathcal{S}$ to the Hilbert transform  H, the
discrete random walk $(B_k)_{k \geqslant 0}$ has to be crafted in a very
unique manner, as we discuss now. We want to build this discrete random walk
so as to obtain the following diagram, where ``convergence'' is meant in a
weak sense defined later.
\[ \begin{array}{lllll}
     & f & \overset{\mathcal{H}}{\longrightarrow} & g & \\
     (\tmop{discretization} / \tmop{convergence}) & \longdownarrow
     \longuparrow &  & \longdownarrow \longuparrow & (\tmop{discretization} /
     \tmop{convergence})\\
     & M_k^f & \overset{\mathcal{S}}{\longrightarrow} & M_k^g & 
   \end{array} . \]
With the very specific process $B$ to be chosen, we want $M^g = \mathcal{S}
M^f$, where
\begin{equation}
  (\mathcal{S} M^f)_k \assign \sum_{l = 1}^k \nabla f (B_{l - 1} (x)) \cdot
  \mathcal{S} \mathd B_l (x) . \label{eq:SMf}
\end{equation}
But rewriting $M^g$ in \eqref{eq:Mfn Mgn} as
\begin{equation}
  M_k^g (x) = \sum_{l = 1}^k \nabla^{\perp} f (B_{l - 1} (x)) \cdot \mathd B_l
  (x) = \sum_{l = 1}^k \nabla f (B_{l - 1} (x)) \cdot \mathd B_l^{\top} (x),
  \label{eq:Mg bis}
\end{equation}
where $\mathd B^{\top}_k$denotes the vector anticlockwise orthogonal to
$\mathd B_k$, and comparing \eqref{eq:SMf} with \eqref{eq:Mg bis} suggests
that the discrete random walk should have the property
\begin{equation}
  \mathcal{S} \mathd B_l {= \mathd B^{\top}_l}  . \label{eq:SdB}
\end{equation}
Finally, notice that \eqref{eq:Mg bis} is the discrete counterpart of
\[ \mathcal{M}_t^g \assign \int_0^{\tau} \nabla f (W_s)^{\perp} \cdot \mathd
   W_s = \int_0^{\tau} \nabla f (W_s) \cdot \mathd W^{\top}_s \]
where $\mathd W^{\top}_s$denotes the vector anticlockwise orthogonal to
$\mathd W_s$.

\begin{remark}
  A fundamental difference between the dyadic Hilbert transform $\mathcal{S}$
  and the classical Haar shift $\mathcal{S}_{\tmop{cl}}$ is revealed by
  equation \eqref{eq:SdB}. The classical shift transforms the values of a
  given function $f$ in a way similar to the Hilbert transform on
  trigonometric functions. Indeed, it transforms a Haar function $h_I$, that
  we can understand as a dyadic counterpart of the sinus function, into a
  dyadic counterpart of the cosinus function.
  
  Equation \eqref{eq:SdB} for the dyadic Hilbert transform  S  has a more
  abstract interpretation, namely that $\mathcal{S}$ acts on the probability
  space by reordering random walks and sending a random walk $(B_k)_{k
  \geqslant 0}$ to $(B_k^{\top})_{k \geqslant 0}$. Similarly, the Hilbert
  transform can be interpreted as sending an occurence $(W_t)_{t \geqslant 0}$
  of the Brownian motion to its rotated version $(W_t^{\top})_{t \geqslant
  0}$.
\end{remark}

\

\section{ random walks, stopping times and scalings}

\

\paragraph{Construction of the discrete random walk}The set of dyadic
intervals $\mathcal{D}$ in the interval $I_0 \assign [0, 1)$ is
\[ \mathcal{D} \assign \left\{ [m 2^{- k}, (m + 1) 2^{- k}) ; \quad 0
   \leqslant k, 0 \leqslant m \leqslant 2^k - 1 \right\} = \bigcup_{k
   \geqslant 0} \mathcal{D}_k \]
where $\mathcal{D}_k \assign \left\{ [m 2^{- k}, (m + 1) 2^{- k}) ; \quad 0
\leqslant m \leqslant 2^k - 1 \right\}$ is the set of dyadic intervals in
generation $k$ and having length $2^{- k}$. We note $\mathcal{D}^- \subset
\mathcal{D}$ the subset of left children and $\mathcal{D}^+ \subset
\mathcal{D}$ the subset of right children. We have therefore $\mathcal{D} = \{
I_0 \} \cup \mathcal{D}^- \cup \mathcal{D}^+$. We also note $\mathcal{D}_k^-
\assign \mathcal{D}_k \cap \mathcal{D}^-$ the left children of intervals of
$\mathcal{D}_k$, and $\mathcal{D}_k^+ \assign \mathcal{D}_k \cap
\mathcal{D}^+$ the right children of intervals of $\mathcal{D}_k$. Let the
sequence $(\varepsilon_k)_{k \geqslant 0}$ of random variables be defined as
\[ \forall k \geqslant 0, \forall x \in [0, 1), \quad \varepsilon_k (x) =
   \left\{ \begin{array}{ll}
     - 1, & x \in \mathcal{D}_{k + 1}^-\\
     1, & x \in \mathcal{D}_{k + 1}^+ .
   \end{array} \right. \]
In other words, if $k \geqslant 0$ and $I \in \mathcal{D}_k$, the random
variable $\varepsilon_k$ takes values $\pm 1$ on $I$ with equal probability.
Equivalently in terms of the Haar functions
\[ \forall k \geqslant 0, \forall x \in [0, 1), \quad \varepsilon_k (x) =
   \sum_{I \in \mathcal{D}_k} \sqrt{| I |} h_I (x) . \]
Let now $\delta > 0$. We build a twodimensional random walk $B_k (x) = (B_k^1
(x), B_k^2 (x))$ defined as
\[ \{ \begin{array}{l}
     B^1_k (x) \assign \sum_{l = 1}^k \tmmathbf{1} (\varepsilon_{l - 1} (x) =
     - 1) \varepsilon_l (x)  \sqrt{2 \delta} \backassign \sum_{l = 1}^k \mathd
     B_l^1 (x)\\
     B^2_k (x) \assign \sum_{l = 1}^k \tmmathbf{1} (\varepsilon_{l - 1} (x) =
     + 1) \varepsilon_l (x)  \sqrt{2 \delta} \backassign \sum_{l = 1}^k \mathd
     B_l^2 (x),
   \end{array} \]
with $B_0 (x) = (0, 0)$ for all $x$. It follows successively, for any $l
\geqslant 1$,
\begin{eqnarray}
  \mathd B^1_l (x) & \assign & \tmmathbf{1} (\varepsilon_{l - 1} (x) = - 1)
  \varepsilon_l (x)  \sqrt{2 \delta} = \sqrt{2 \delta} \sum_{I \in
  \mathcal{D}_l^-} \sqrt{| I |} h_I (x) \nonumber\\
  \mathcal{S} \mathd B^1_l (x) & \assign & \sqrt{2 \delta} \sum_{I \in
  \mathcal{D}_l^-} \sqrt{| I |}  \mathcal{S} h_I (x) = \sqrt{2 \delta} \sum_{I
  \in \mathcal{D}_l^+} \sqrt{| I |}  (- h_I (x)) = - \mathd B^2_l (x)
  \nonumber\\
  \mathd B^2_l (x) & \assign & \tmmathbf{1} (\varepsilon_{l - 1} (x) = + 1)
  \varepsilon_l (x)  \sqrt{2 \delta} = \sqrt{2 \delta} \sum_{I \in
  \mathcal{D}_l^+} \sqrt{| I |} h_I (x) \nonumber\\
  \mathcal{S} \mathd B^2_l (x) & \assign & \sqrt{2 \delta} \sum_{I \in
  \mathcal{D}_l^+} \sqrt{| I |}  \mathcal{S} h_I (x) = \sqrt{2 \delta} \sum_{I
  \in \mathcal{D}_l^-} \sqrt{| I |}  (+ h_I (x)) = + \mathd B^1_l (x) .
  \nonumber
\end{eqnarray}
As a conclusion, denoting $\mathd B_l \assign (\mathd B^1_l, \mathd B^2_l)$ a
twodimensional increment, we have $\mathcal{S} \mathd B_l = \mathd B_l^{\top}$
as desired. We define now the discrete martingales associated to $f$ and $g =
\mathcal{H} f$ respectively, as
\[ \forall k, \quad M^f_k \assign f (B_0) + \sum_{l = 1}^k \nabla f (B_{l -
   1}) \cdot \mathd B_l, \quad M^g_k \assign g (B_0) + \sum_{l = 1}^k \nabla g
   (B_{l - 1}) \cdot \mathd B_l, \]
and observe:

\begin{lemma}
  \label{L: Lp estimate for Mng}Let $M^f$ and $M^g$ as above. We have
  $\mathcal{S} M^f = M^g$ and therefore also
  \[ \forall k, \quad \| M_k^g \|_p \leqslant \| \mathcal{S} \|_{p \rightarrow
     p} \| M_k^f \|_p . \]
\end{lemma}

\begin{proof}
  Since $\mathcal{S} \mathd B_l = \mathd B_l^{\top},$ $\nabla g =
  \nabla^{\perp} f$ and $g (B_0) = 0$, we have successively, for all $k
  \geqslant 1$,
  \begin{eqnarray*}
    (\mathcal{S} M^f)_k & \assign & \mathcal{S} \left( f (B_0) + \sum_{l =
    1}^k \nabla f (B_{l - 1}) \cdot \mathd B_l \right) \assign 0 + \sum_{l =
    1}^k \nabla f (B_{l - 1}) \cdot \mathcal{S} \mathd B_l\\
    & = & \sum_{l = 1}^k \nabla f (B_{l - 1}) \cdot \mathd B_l^{\top} =
    \sum_{l = 1}^k \nabla f (B_{l - 1})^{\perp} \cdot \mathd B_l\\
    & = & M_k^g .
  \end{eqnarray*}
  Hence the result.
\end{proof}

In order to obtain from the inequality above an inequality for the continuous
martingales $\mathcal{M}^f$ and $\mathcal{M}^g$, we need to prove in a
suitable sense the convergence of the discrete martingales $M^f$ and $M^g$
towards their continuous counterparts $\mathcal{M}^f$ and $\mathcal{M}^g$.
Since we are only interested in norms of those processes, what we need to
obtain is some so--called weak convergence estimate, see e.g. {\tmname{Talay}}
or {\tmname{Kloeden}}--{\tmname{Platen}} {\cite{Tal1986n,KloPla1992}}.
Unfortunately, our situation does not exactly fit those expositions, for the
following reasons:
\begin{itemizedot}
  \item We consider randomly stopped processes as opposed to processes running
  on a prescribed, deterministic, interval of time.
  
  \item The discrete random walk is \tmtextit{not} a Markov process, rather a
  so--called Markov process with memory
  
  \item The quadratic covariations of the discrete process are never
  converging towards that of their continuous counterparts, since they can be
  null half of the time.
  
  \item We do not compare only discrete random walks and their continuous
  counterparts but also their martingale transforms $M^f$ and $M^g$.
\end{itemizedot}
We define in the next section some auxiliary processes and suitable stopped
random walks. Only then will we be able to state in a precise manner at the
end of this section the convergence result Theorem \ref{T: convergence} suited
to our needs.

\

\paragraph{Scalings and stopped random walks}Let $W \assign (W_t)_{t
\geqslant 0}$ the standard two dimensional Brownian process started at the
origin. Let again $\tau$ the stopping time
\[ \tau \assign \inf \{ t > 0 ; W_t \nin \mathbb{D} \}, \]
that is the first time of exit of the unit disc $\mathbb{D}$. We denote by
$W^{\tau} = (W^{\tau}_t)_{t \in [0, \infty)} \assign (W_{t \wedge \tau})_{t
\in [0, \infty)}$ the corresponding stopped process. Finally, given $t \in
\mathbb{R}$, $x \in \mathbb{D}$, we note $(W^{\tau, t, x}_{t + s})_{s
\geqslant 0}$ the process $W^{\tau}$ started at $x$ at time $t$.

\

Let $T > 0$ a fixed time, and $N \in \mathbb{N}$ large. Define $\delta$ a
small time--step such that $T = N^5 \delta$. In order to denote the
corresponding discrete times, we will use the indices $k, l \in [0, N^5]$,
typically $t_k \assign k \delta$, $t_l \assign l \delta$. We introduce a
larger time-step $\theta$ defined as $\theta \assign N \delta$. Notice that $T
= N^4 \theta$, so that also $\theta$ tends to zero as $N$ goes to infinity for
fixed $T$. The discrete times corresponding to this time--step will use
indices $n, m \in [0, N^4]$, typically $t_n \assign n \theta$, $t_m \assign m
\theta$. Given the discrete random walk $B$ defined above, we define a new
discrete random walk $X$ by sampling $B$ at times that are multiples of
$\theta = N \delta$, therefore with indinces that are multiples of $N$,
\[ \forall n \geqslant 0, \quad X_n \assign B_{n N} . \]
Notice further that
\[ \mathd X_n \assign X_n - X_{n - 1} = B_{n N} - B_{(n - 1) N} = \sum_{l =
   1}^N \mathd B_{(n - 1) N + l} . \]
In order to stop $X$ before it leaves the unit disc, we set $\varepsilon = 1 /
N$ and define the discrete stopping time
\[ \tau_{\varepsilon} \assign \inf \{ t_n ; X_n \nin (1 - \varepsilon)
   \mathbb{D} \} . \]
We will denote by $n_{\varepsilon}$ the random index such that
$\tau_{\varepsilon} \assign n_{\varepsilon} \theta$. \ From the definition of
$X$ we have for all $n$,
\[ | X_n - X_{n - 1} | \assign | \mathd X_n | \leqslant \sum_{l = 1}^N |
   \mathd B_{(n - 1) N + l} | \leqslant N \sqrt{2 \delta} = \sqrt{2 T} N^{- 3
   / 2} . \]
Therefore $| \mathd X_n | \leqslant \varepsilon$ for $N$ large enough. In
particular, before a stopping time we have $X_{\tau_{\varepsilon} - \theta}
\in (1 - \varepsilon) \mathbb{D}$ which implies that the stopped process
$X^{\tau_{\varepsilon}}$ always remains in the unit disc $\mathbb{D}$.
Moreover at stopping time, the random walk $X_{\tau_{\varepsilon}}$ lies in
the band $\mathbb{D} \backslash (1 - \varepsilon) \mathbb{D}$ of width
$\varepsilon$ near the interior boundary $\partial \mathbb{D}$ of the unit
disc. We note $X^{\tau_{\varepsilon}}$ the corresponding stopped process.
Finally we will denote by $k_{\varepsilon}$ the random index such that
$\tau_{\varepsilon} \assign k_{\varepsilon} \delta$, or equivalently
$k_{\varepsilon} \assign N n_{\varepsilon}$. We note $B^{\tau_{\varepsilon}}$
the process $B$ stopped at $\tau_{\varepsilon}$.

Notice that the discrete martingales $(M_k^f)_{k \in [0, N^5]}$ and
$(M_k^g)_{k \in [0, N^5]}$ are martingale transforms of the discrete random
walk $(B_k)_{k \in [0, N^5]}$. We will aslo consider a sampled version of
$(M_k^f)_{k \in [0, N^5]}$, namely $(M_n^f)_{n \in [0, N^4]}$, where $M_n^f
\assign M^f_k$ for $k = n N$. We use the same notation for both processes.
Which one is meant will be clear from the context. Similarly we note
$(\mathcal{F}_k)_{k \in [0, N^5]}$ the filtration associated to the discrete
random walk $(B_k)_{k \in [0, N^5]}$ and $(\widetilde{\mathcal{F}}_n)_{n \in
[0, N^4]}$, with $\widetilde{\mathcal{F}}_n \assign \mathcal{F}_{n N}$, the
filtration associated to the coarse random walk $(X_n)_{n \in [0, N^4]}$.

\

\paragraph{Notations}We assume without loss of generality that $f \in L^p
(\partial \mathbb{D}, X)$ and its harmonic extension (also noted $f$) $f \in
L^p (\mathbb{D}, X)$ are smooth Frechet differentiable functions, that is $f
\in \mathcal{C}^k$ for all $k \geqslant 0$. We note $| f (x) | \assign | f (x)
|_X$ the Banach space norm of $f (x)$. The $L^p$--norm of $f$ is defined on
the circle in the usual way
\[ \| f \|_p \assign \| f \|_{L^p (\partial \mathbb{D}, X)} \assign \left(
   \int_{\partial \mathbb{D}} | f (x) |^p \mathd x \right)^{1 / p} . \]
However if $Y$ is an $L^p$ integrable random variable on the probability space
$(\Omega, \mathbb{P})$, the stochastic $L^p$ norm is defined as
\[ \| Y \|_p \assign \mathbb{E} (| Y |^p)^{1 / p} . \]
If $Y$ is real valued, then $| Y |$ denotes the absolute value of $Y$ whereas
if $Y$ is $X$--valued, then $| Y | \assign | Y |_X$ denotes the Banach space
norm of $Y$.

Now, if $f \assign f (x_1, \ldots, x_m)$ is a $X$--valued function of $m$
variables defined on the open set $U \subset \mathbb{R}^m$, we note $\mathD f
\assign (\partial_1 f, \ldots, \partial_m f)$ its derivatives in the Frechet
sense, where $\mathD f : U \times \mathbb{R}^m \rightarrow X$ is continuous.
Further, given a $m$--multiindex $\alpha \assign (\alpha_1, \alpha_2, \ldots,
\alpha_m)$ with $| \alpha | = k$ we note as usual $\mathD^{\alpha} f \assign
\partial_1^{\alpha_1} \ldots \partial_m^{\alpha_m} f$ its partial derivatives
of order $k$ in the Frechet sense, where $\mathD^{\alpha} f : U \times
(\mathbb{R}^m)^k \rightarrow X$ is continuous. Finally,using again
multiindices, monomials of the from $x^{\alpha}$, where $x \assign (x_1,
\ldots, x_m)$, are a shorthand for $x^{\alpha} \assign x_1^{\alpha_1} \ldots
x_m^{\alpha_m}$.

The dual space of $X$ is noted $X^{\ast}$, and the dual space of $L^p (U, X)$
is noted $L^q (U, X^{\ast})$ with $1 / p + 1 / q = 1$.

For the convergence results, our main parameters are $\varepsilon > 0$, a
small number, and $T > 0$, a large number. We note $c (\varepsilon)$ a generic
function tending to $0$ uniformly in $T$ when $\varepsilon$ goes to zero, $c_T
(\varepsilon)$ a generic function tending to $0$ when $\varepsilon$ goes to
zero for any fixed $T$, $c (T)$ a generic function tending to $0$ when $T$
goes to infinity, uniformely in $\varepsilon$. It is implicit that those
functions all depend on the fixed function $f$ and its derivatives. Further
dependences will be mentionned when needed.

\

\paragraph{Convergence result}In order to prove Theorem \ref{T: H less than
S}, we need the following convergence result of the discrete martingales
towards their continuous counterparts.

\begin{theorem}[Convergence of $L^p$ norms of martingales]
  \label{T: convergence}Let $f$ as above. We have
  \[ \lim_{T \rightarrow \infty} \lim_{\varepsilon \rightarrow 0} \mathbb{E} |
     M_T^f |^p =\mathbb{E} | \mathcal{M}^f_{\infty} |^p \]
\end{theorem}

A key ingredient is the the notion of weak consistency presented in the next
Section \ref{S: weak consistency}. This allows us to prove auxiliary
convergence results in Section \ref{S: auxiliary convergence results}. Finally
Section \ref{S: main results} is devoted to the proofs of the main results,
namely Theorem \ref{T: convergence} and Theorem \ref{T: H less than S}.

\section{Weak consistency and moments estimates.}\label{S: weak consistency}

Convergence results will be obtained by proving the weak consistency of the
(sampled) stopped discrete random walk $X^{\tau_{\varepsilon}}$ with the
continuous stopped twodimensional Brownian process $W^{\tau}$. Due to the
stopping process, we can not rely on the standard definition of weak
consistency based on discrete and continuous stochastic equations with
prescribed coefficients depending smoothly on the process alone as used in
e.g. {\cite{Tal1986n,KloPla1992}}. The definition adapted to our situation
simply reads:

\begin{definition}
  We say that $X^{\tau_{\varepsilon}}$ is weakly consistent with $W^{\tau}$,
  iff there exists a function $c \assign c (\varepsilon)$ tending to zero when
  $\varepsilon$ goes to zero, such that for all $n$, all $2$--multiindex
  $\alpha \assign (\alpha_1, \alpha_2)$ with $| \alpha | = 2$, there holds
  \[ \left| \mathbb{E} (X^{\tau_{\varepsilon}}_{n + 1} -
     X^{\tau_{\varepsilon}}_n | \widetilde{\mathcal{F}}_n) -\mathbb{E} \left(
     W^{\tau, \tau_{\varepsilon}, X_n^{\tau_{\varepsilon}}}_{t_{n + 1}} -
     W_{t_n}^{\tau, \tau_{\varepsilon}, X_n^{\tau_{\varepsilon}}} |
     \widetilde{\mathcal{F}}_n \right) \right| \leqslant \theta c
     (\varepsilon), \]
  \[ \left| \mathbb{E} ((X^{\tau_{\varepsilon}}_{n + 1} -
     X^{\tau_{\varepsilon}}_n)^{\alpha} | \widetilde{\mathcal{F}}_n)
     -\mathbb{E} \left( \left( W^{\tau, \tau_{\varepsilon},
     X_n^{\tau_{\varepsilon}}}_{t_{n + 1}} - W_{t_n}^{\tau,
     \tau_{\varepsilon}, X_n^{\tau_{\varepsilon}}} \right)^{\alpha} |
     \widetilde{\mathcal{F}}_n \right) \right| \leqslant \theta c
     (\varepsilon) . \]
\end{definition}

We can now state

\begin{lemma}
  \label{L: weak consistency} The discrete stopped process
  $X^{\tau_{\varepsilon}}$ is weakly consistent with the continuous stopped
  process $W^{\tau}$.
\end{lemma}

\begin{proof}
  The weak consistency is a straightforward consequence of the next two
  moments lemmas.
\end{proof}

\begin{lemma}[Discrete Moments]
  \label{L: discrete moments}Let $X^{\tau_{\varepsilon}}$ as above. There
  exists a function $c \assign c (\varepsilon)$ tending to zero when
  $\varepsilon$ goes to zero, such that for all $n < n_{\varepsilon}$, all
  $2$--multiindex $\alpha \assign (\alpha_1, \alpha_2)$ with $| \alpha | = 2$,
  there holds
  \[ \mathbb{E} \left( \mathd X^{\tau_{\varepsilon}}_{n + 1} |
     {\widetilde{\mathcal{F}}_n}  \right) = 0 \]
  \[ \mathbb{E} ((\mathd X^{\tau_{\varepsilon}}_{n + 1})^{\alpha} |
     \widetilde{\mathcal{F}}_n) =\tmmathbf{1} (\alpha_1 \neq \alpha_2) \theta
     (1 + c (\varepsilon)), \]
  Moreover for all $p \geqslant 2$, all $i = 1, 2$, there holds
  \[ \mathbb{E} (| \mathd X_{n + 1}^{\tau_{\varepsilon}, i} |^p |
     \widetilde{\mathcal{F}}_n) \lesssim \theta^{p / 2} . \]
\end{lemma}

Notice that $\tmmathbf{1} (\alpha_1 \neq \alpha_2) = 1$ if $\alpha \in \{ (2,
0), (0, 2) \}$ and $\tmmathbf{1} (\alpha_1 \neq \alpha_2) = 0$ if $\alpha =
(1, 1)$.

\begin{lemma}[Continuous moments]
  \label{L: continuous moments} Let $W^{\tau}$ and $W^{\tau, t, x}$ as above.
  There exists a function $c \assign c_T (\varepsilon)$ tending to zero when
  $\varepsilon$ goes to zero for all fixed $T$, such that for all
  $2$--multiindex $\alpha \assign (\alpha_1, \alpha_2)$ with $| \alpha | = 2$,
  there holds
  \[ \forall t \geqslant 0, \forall x \in \mathbb{D}, \quad \mathbb{E} (W_{t +
     \theta}^{\tau, t, x} - W_t^{\tau, t, x}) = 0, \]
  \[ \forall t \geqslant 0, \forall x \in (1 - \varepsilon) \mathbb{D}, \quad
     \mathbb{E} (W_{t + \theta}^{\tau, t, x} - W_t^{\tau, t, x})^{\alpha}
     =\tmmathbf{1} (\alpha_1 \neq \alpha_2) \theta (1 + c_T (\varepsilon)) .
  \]
  \[ \forall t \geqslant 0, \forall x \in \mathbb{D}, \forall p \geqslant 2,
     \quad \mathbb{E} (| W_{t + \theta}^{\tau, t, x} - W_t^{\tau, t, x} |^p)
     \lesssim \theta^{p / 2} . \]
\end{lemma}

\begin{proof*}{Proof of Lemma \ref{L: discrete moments}}
  {\dueto{Discrete Moments}}For simplicity, we omit the superfix
  $\tau_{\varepsilon}$ in the discrete stopped processes
  $X^{\tau_{\varepsilon}}$ and $B^{\tau_{\varepsilon}}$. Recall that
  $(\mathcal{F}_k)_{k \in [0, N^5]}$ is the filtration associated to the
  discrete random walk $(B_k)_{k \in [0, N^5]}$, and
  $(\widetilde{\mathcal{F}}_n)_{n \in [0, N^4]}$, with
  $\widetilde{\mathcal{F}}_n \assign \mathcal{F}_{n N}$, the filtration
  associated to the coarse random walk $(X_n)_{n \in [0, N^4]}$. Recall
  finally that $X_n \assign B_{n N}$, and therefore $\mathd X_{n + 1} \assign
  X_{n + 1} - X_n = \sum_{l = 1}^N \mathd B_{n N + l}$. We are only interested
  in increments occuring before stopping, otherwise the estimate is trivial,
  all increments being zero after stopping. We have for all $n \leqslant N^4$,
  \[ \mathbb{E} (\mathd X_{n + 1} | \widetilde{\mathcal{F}}_n) = \sum_{l =
     1}^N \mathbb{E} (\mathd B_{n N + l} | \mathcal{F}_{n N}) = 0 \]
  since $B$ is a martingale. Now for the second order moments, let $\alpha =
  (\alpha_1, \alpha_2)$ a multiindex with $| \alpha | = 2$. We estimate for
  example the variance of the first coordinate
  \begin{eqnarray*}
    \mathbb{E} ((\mathd X^1_{n + 1})^2 | \widetilde{\mathcal{F}}_n) & \assign
    & \mathbb{E} \left( \sum_{l, l' = 1}^N \mathd B_{n N + l}^1 \mathd B_{n N
    + l'}^1 | \mathcal{F}_{n N} \right) = \sum_{l = 1}^N \mathbb{E} ((\mathd
    B_{n N + l}^1)^2 | \mathcal{F}_{n N})
  \end{eqnarray*}
  where we used that $B$ is a martingale. We have in the sum above for $l =
  1$,
  \begin{eqnarray*}
    \mathbb{E} ((\mathd B_{n N + 1}^1)^2 | \mathcal{F}_{n N}) & = & \mathbb{E}
    \left( \tmmathbf{1} (\varepsilon_{n N} = - 1)^2 \varepsilon_{n N + 1}^2 
    \left( \sqrt{2 \delta} \right)^2 | \mathcal{F}_{n N} \right) =\tmmathbf{1}
    (\varepsilon_{n N} = - 1) 2 \delta .
  \end{eqnarray*}
  For the other summands, with $l \geqslant 2$, we have
  \begin{eqnarray*}
    \mathbb{E} ((\mathd B_{n N + l}^1)^2 | \mathcal{F}_{n N}) & = & \mathbb{E}
    \left( \tmmathbf{1} (\varepsilon_{n N + l - 1} = - 1)^2 \varepsilon_{n N +
    l}^2  \left( \sqrt{2 \delta} \right)^2 | \mathcal{F}_{n N} \right)\\
    & = & 2 \delta \quad \mathbb{E} (\tmmathbf{1} (\varepsilon_{n N + l - 1}
    = - 1)  | \mathcal{F}_{n N}) = 2 \delta \left[ \frac{1}{2} \cdot 0 +
    \frac{1}{2} \cdot 1 \right] = \delta
  \end{eqnarray*}
  Since $\theta = N \delta$, we have, for $\alpha = (2, 0)$,
  \begin{eqnarray*}
    \mathbb{E} ((\mathd X^1_{n + 1})^2 | \widetilde{\mathcal{F}}_n) & = &
    [\tmmathbf{1} (\varepsilon_{n N} = - 1) 2 \delta + (N - 1) \delta] =
    \theta (1 + c (\varepsilon)),
  \end{eqnarray*}
  and the same estimate holds for the second coordinate. Finally since $\mathd
  B^1_n (x) \mathd B^2 (x) = 0$ for all $n$, all $x$, it follows $\mathbb{E}
  (\mathd X^1_{n + 1} \mathd X^2_{n + 1} | \widetilde{\mathcal{F}}_n) = 0$ for
  all $n$.
  
  For the higher order moments, if $| \alpha | = p$, $\mathbb{E} (\mathd X_{n
  + 1}^{\alpha} | \widetilde{\mathcal{F}}_n) \lesssim \mathbb{E} (| \mathd
  X_{n + 1}^1 |^p | \widetilde{\mathcal{F}}_n) +\mathbb{E} (| \mathd X_{n +
  1}^2 |^p | \widetilde{\mathcal{F}}_n)$. For integer moments, estimate for
  example
  
  \begin{align}
    \mathbb{E} (| \mathd X_{n + 1}^1 |^{2 p} | \widetilde{\mathcal{F}}_n) & =
    \sum_{l_1, \ldots, l_{2 p}} \mathbb{E} (\mathd B^1_{n N + l_1} \ldots
    \mathd B_{n N + l_{2 p}}^1 | \widetilde{\mathcal{F}}_n) \nonumber\\
    & \leqslant \frac{(2 p) !}{2} \sum_{l_1, l_3, \ldots l_{2 p - 1}}
    \mathbb{E} ((\mathd B^1_{n N + l_1})^2 (\mathd B^1_{n N + l_3})^2 \ldots
    (\mathd B_{n N + l_{2 p - 1}}^1)^2 | \widetilde{\mathcal{F}}_n)
    \nonumber\\
    & \lesssim N^p \delta^p \lesssim \theta^p . \nonumber
  \end{align}
  
  By H{\"o}lder, we get for any $p \geqslant 2$, $\mathbb{E} (| \mathd X_{n +
  1}^1 |^p | \widetilde{\mathcal{F}}_n) \lesssim \theta^{p / 2}$.
\end{proof*}

\begin{proof*}{Proof of Lemma \ref{L: continuous moments}}
  {\dueto{Continuous Moments}}The first statement is obvious. For the second
  statement, notice that $x \in (1 - \varepsilon) \mathbb{D}$ implies
  $\tmop{dist} (x, \partial \mathbb{D}) \geqslant \varepsilon$. We now take
  advantage of the fact that the standard mean deviation of the non--stopped
  Brownian motion $W_{t + \theta}^{t, x} - W_t^{t, x}$ over one time step
  $\theta$  is $\sigma_{\theta} = \sqrt{\theta} = \sqrt{T / N^4} = \varepsilon
  \sqrt{T} / N$. Therefore $\sigma_{\theta} \ll \varepsilon$, i.e.
  $\sigma_{\theta} = c_T (\varepsilon) \varepsilon .$ This means that for $x
  \in (1 - \varepsilon) \mathbb{D}$, we have $\tmop{dist} (x, \partial
  \mathbb{D}) \geqslant \varepsilon$.In other words the point $x$ is ``very
  far away'' from the boundary as compared to the distance the Brownian motion
  can diffuse. More precisely, as a consequence of standard estimates of first
  hitting times of the Brownian motion, we have for such a $x$ and the
  corresponding stopped process $W_t^{\tau, t, x}$, that
  \[ \mathbb{P} (t + \theta > \tau) = c (\sigma_{\theta} / \varepsilon) = c
     (c_T (\varepsilon)) = c_T (\varepsilon) . \]
  Note now $(W^{\tau, t, x, i}_s)_{i = 1, 2}$ the two components of the
  Brownian motion. Consider for example the second moment of the first
  coordinate. Using It{\^o} formula and taking expectation yields
  
  \begin{align}
    \mathbb{E} (W_{t + \theta}^{\tau, t, x, 1} - W_t^{\tau, t, x, 1})^2 & =
    \frac{1}{2} \mathbb{E} \int_t^{(t + \theta) \wedge \tau} \mathd [W^{\tau,
    t, x, 1}, W^{\tau, t, x, 1}]_s \nonumber\\
    & = \frac{1}{2} \mathbb{E} \left( \int_t^{(t + \theta)} \mathd [W^{\tau,
    t, x, 1}, W^{\tau, t, x, 1}]_s  | t + \theta \leqslant \tau \right)
    \mathbb{P} (t + \theta \leqslant \tau) \nonumber\\
    & + \frac{1}{2} \mathbb{E} \left( \int_t^{(t + \theta)} \mathd [W^{\tau,
    t, x, 1}, W^{\tau, t, x, 1}]_s  | t + \theta > \tau \right) \mathbb{P} (t
    + \theta > \tau) \nonumber\\
    & \backassign \theta \delta_{\alpha_1 \alpha_2} \mathbb{P} (t + \theta
    \leqslant \tau) + A (t, x, \theta) \mathbb{P} (t + \theta > \tau),
    \nonumber
  \end{align}
  
  where clearly $0 \leqslant A (t, x, \theta) \leqslant \theta$ for all $(t,
  x)$. This yields the second statement since $\mathbb{P} (t + \theta > \tau)
  = c_T (\varepsilon)$ and $\mathbb{P} (t + \theta \leqslant \tau) = 1 - c_T
  (\varepsilon)$ in the case $\alpha = (2, 0)$. The case $\alpha = (0, 2)$ is
  similar. The case $\alpha = (1, 1)$ is trivial since $\mathd [W^{\tau, t, x,
  1}, W^{\tau, t, x, 2}]_s = 0$. The third statement of the Lemma follows the
  same lines using the known higher moments for the Brownian motion.
\end{proof*}

\section{Auxiliary convergence results}\label{S: auxiliary convergence
results}

Our goal is to prove the following two convergence results:

\begin{lemma}
  \label{L: weak convergence} (Weak convergence) Let $T > 0$. Let $\psi$
  harmonic on $\mathbb{D}$ with $\psi$ smooth on $\partial \mathbb{D}$. Assume
  weak consistency. Then we have
  \[ \mathbb{E} \psi (X_T^{\tau_{\varepsilon}}) =\mathbb{E} \psi (W_T^{\tau})
     + c_{\psi, T} (\varepsilon) + c (T) \]
\end{lemma}

\begin{lemma}
  \label{L: convergence discrete martingale transforms}(Weak convergence of
  discrete martingale transforms) Let $T > 0$. Let $f$ as above.
  \[ \| f (X_T^{\tau_{\varepsilon}}) - M_T^f \|_p = c_T (\varepsilon) \]
\end{lemma}

\begin{proof*}{Proof of Lemma \ref{L: weak convergence}}
  Let $\psi$ harmonic on $\mathbb{D}$, with $\psi$ smooth on $\partial
  \mathbb{D}$. We first split
  \[ \mathbb{E} \psi (X_T^{\tau_{\varepsilon}}) =\mathbb{E} (\psi
     (X_T^{\tau_{\varepsilon}}) | \tau_{\varepsilon} \leqslant T) \mathbb{P}
     (\tau_{\varepsilon} \leqslant T) +\mathbb{E} (\psi
     (X_T^{\tau_{\varepsilon}}) | \tau_{\varepsilon} > T) \mathbb{P}
     (\tau_{\varepsilon} > T) \]
  We claim that the second term is small uniformly w.r.t. $\varepsilon$ when
  $T$ is large. Indeed, by definition of $\tau_{\varepsilon}$, this term
  collects the contribution of those trajectories that remained in the disc
  $(1 - \varepsilon) \mathbb{D}$ during the whole interval of time $[0, T]$.
  We claim that this is small for $T$ large. Indeed, let $\tilde{B} =
  (\tilde{B}^1, \tilde{B}^2)$ the rotation of angle $\pi / 4$ of $B$, i.e.
  $\tilde{B}^1 \assign (B^1 + B^2) / \sqrt{2}$, and $\tilde{B}^2 \assign (B^2
  - B^1) / \sqrt{2}$. It follows
  \[ \tilde{B}^1_k (x) = \sum_{l = 1}^k \varepsilon_l (x)  \sqrt{\delta},
     \quad \tilde{B}^2_k (x) = \sum_{l = 1}^k \varepsilon_{l - 1} (x)
     \varepsilon_l (x)  \sqrt{\delta}, \]
  that is both $\tilde{B}^1$ and $\tilde{B}^2$ are (non independent) standard
  centered discrete random walks. Let $\tilde{\tau}^1$ and $\tilde{\tau}^2$
  the first exit times
  \[ \tilde{\tau}^1 \assign \inf \{ t_k ; | \tilde{B}_{t_k}^1 | \geqslant 1
     \}, \quad \tilde{\tau}^2 \assign \inf \{ t_k ; | \tilde{B}_{t_k}^2 |
     \geqslant 1 \} . \]
  But
  \begin{eqnarray*}
    \tau_{\varepsilon} > T & \Leftrightarrow & X_n \in (1 - \varepsilon)
    \mathbb{D}, \quad n \in [0, N^4]\\
    & \Rightarrow & B_{n N} \in (1 - \varepsilon) \mathbb{D}, \quad n \in [0,
    N^4]\\
    & \Rightarrow & B_k \in \mathbb{D}, \quad k \in [0, N^5]\\
    & \Rightarrow & \tilde{\tau}^1 > T \infixand \tilde{\tau}^2 > T,
  \end{eqnarray*}
  where we have used that $| B_k - B_{n N} | \leqslant N \sqrt{\delta} \ll
  \varepsilon$ for $k \in [(n - 1) N + 1, \ldots, n N]$. In particular,
  \[ \mathbb{P} (\tau_{\varepsilon} > T) \leqslant \mathbb{P} (\tilde{\tau}^1
     > T) = c (T) . \]
  The last equality is a consequence of first hitting time estimates of
  standard centered discrete random walks, see e.g. {\tmname{Lawler}}
  {\cite{Law2010a}}. The function $\psi$ being bounded on $\mathbb{D}$, we
  have also $\mathbb{E} (\psi (X_T^{\tau_{\varepsilon}}) | \tau_{\varepsilon}
  > T) \mathbb{P} (\tau_{\varepsilon} > T) = c (T)$. Similarly for the second
  term, $\mathbb{E} (\psi (X_T^{\tau_{\varepsilon}}) | \tau_{\varepsilon}
  \leqslant T) \mathbb{P} (\tau_{\varepsilon} \leqslant T) =\mathbb{E} (\psi
  (X_T^{\tau_{\varepsilon}}) | \tau_{\varepsilon} \leqslant T) + c_{\psi}
  (T)$. On the other hand, since $\psi$ is harmonic, we have immediately
  \[ \mathbb{E} \psi (W_T^{\tau}) = \psi (W_0^{\tau}) = \psi (0, 0), \]
  so that
  \[ \mathbb{E} \psi (X_T^{\tau_{\varepsilon}}) -\mathbb{E} \psi (W_T^{\tau})
     =\mathbb{E} (\psi (X_T^{\tau_{\varepsilon}}) | \tau_{\varepsilon}
     \leqslant T) - \psi (0, 0) + c (T) =\mathbb{E} (\psi
     (X_T^{\tau_{\varepsilon}}) - \psi (X_0^{\tau_{\varepsilon}}) |
     \tau_{\varepsilon} \leqslant T) + c_{\psi} (T) . \]
  Now since $\mathbb{E} (\psi (W_{t_{n + 1}}^{\tau, t_n, x})) = \psi
  (W_{t_n}^{\tau, t_n, x}) = \psi (x)$ for all $x \in \mathbb{D}$, we have
  
  \begin{align}
    \mathbb{E} (\psi (X_T^{\tau_{\varepsilon}}) - \psi
    (X_0^{\tau_{\varepsilon}})) & =\mathbb{E} \sum_{n = 1}^{n_{\varepsilon}}
    (\psi (X_{t_n}^{\tau_{\varepsilon}}) - \psi (X_{t_{n -
    1}}^{\tau_{\varepsilon}})) \nonumber\\
    & =\mathbb{E} \sum_{n = 1}^{n_{\varepsilon}} (\psi
    (X_{t_n}^{\tau_{\varepsilon}}) - \psi (X_{t_{n -
    1}}^{\tau_{\varepsilon}})) - \left( \psi \left( W_{t_n}^{\tau, t_{n - 1},
    X_{t_{n - 1}}^{\tau_{\varepsilon}}} \right) - \psi (X_{t_{n -
    1}}^{\tau_{\varepsilon}}) \right) \nonumber\\
    & =\mathbb{E} \sum_{n = 1}^{n_{\varepsilon}} \left[ D \psi (X_{t_{n -
    1}}^{\tau_{\varepsilon}}) \cdot \left\{ (X_{t_n}^{\tau_{\varepsilon}} -
    X_{t_{n - 1}}^{\tau_{\varepsilon}}) - \left( W_{t_n}^{\tau, t_{n - 1},
    X_{t_{n - 1}}^{\tau_{\varepsilon}}} - X_{t_{n - 1}}^{\tau_{\varepsilon}}
    \right) \right\} \right. \nonumber\\
    & \qquad + \frac{1}{2} \sum_{| \alpha | = 2} D^{\alpha} \psi (X_{t_{n -
    1}}^{\tau_{\varepsilon}}) \cdot \left\{ (X_{t_n}^{\tau_{\varepsilon}} -
    X_{t_{n - 1}}^{\tau_{\varepsilon}})^{\alpha} - \left( W_{t_n}^{\tau, t_{n
    - 1}, X_{t_{n - 1}}^{\tau_{\varepsilon}}} - X_{t_{n -
    1}}^{\tau_{\varepsilon}} \right)^{\alpha} \right\} \nonumber\\
    & \qquad + \left. R (X_{t_{n - 1}}^{\tau_{\varepsilon}},
    X_{t_n}^{\tau_{\varepsilon}}) - R \left( X_{t_{n -
    1}}^{\tau_{\varepsilon}}, W_{t_n}^{\tau, t_{n - 1}, X_{t_{n -
    1}}^{\tau_{\varepsilon}}} \right) \right] \nonumber\\
    & \backassign A + B + C, \nonumber
  \end{align}
  
  where $R (x, y)$ is the Taylor rest
  \[ R (x, y) \assign \frac{1}{3!} \sum_{| \alpha | = 3} D^{\alpha} \psi (x +
     \theta_{x, y} (y - x)) \cdot (y - x)^{\alpha}, \quad x, y \in \mathbb{D},
     \theta_{x, y} \in [0, 1] . \]
  Using the weak consistency of $X^{\tau_{\varepsilon}}$ with $W^{\tau}$, we
  get, recalling $T = N^4 \theta$,
  \begin{eqnarray*}
    | A | & \leqslant & \mathbb{E} \sum_{n = 1}^{n_{\varepsilon}} \left[ | D
    \psi (X_{t_{n - 1}}^{\tau_{\varepsilon}}) |  \left| \mathbb{E}
    (X_{t_n}^{\tau_{\varepsilon}} - X_{t_{n - 1}}^{\tau_{\varepsilon}} |
    \mathcal{F}_{n - 1} \nobracket) -\mathbb{E} \left( W_{t_n}^{\tau, t_{n -
    1}, X_{t_{n - 1}}^{\tau_{\varepsilon}}} - X_{t_{n -
    1}}^{\tau_{\varepsilon}} | \mathcal{F}_{n - 1} \nobracket \right) \right|
    \right.\\
    & \leqslant & \| D \psi \|_{\infty} \quad \mathbb{E} \sum_{n =
    1}^{n_{\varepsilon}} \theta c (\varepsilon) \leqslant \| D \psi
    \|_{\infty} N^4 \theta c (\varepsilon) \lesssim c_{\psi} (\varepsilon) T =
    c_{\psi, T} (\varepsilon) .
  \end{eqnarray*}
  Similarly for the second moments, since $X_{t_{n - 1}}^{\tau_{\varepsilon}}
  \in (1 - \varepsilon) \mathbb{D}$,
  
  \begin{align}
    | B | & \leqslant \mathbb{E} \sum_{n = 1}^{n_{\varepsilon}} \sum_{| \alpha
    | = 2} | D^{\alpha} \psi (X_{t_{n - 1}}^{\tau_{\varepsilon}}) | \times
    \nonumber\\
    & \qquad \left| \mathbb{E} ((X_{t_n}^{\tau_{\varepsilon}} - X_{t_{n -
    1}}^{\tau_{\varepsilon}})^{\alpha} | \mathcal{F}_{n - 1} \nobracket)
    -\mathbb{E} \left( \left( W_{t_n}^{\tau, X_{t_{n -
    1}}^{\tau_{\varepsilon}}, t_n} - X_{t_{n - 1}}^{\tau_{\varepsilon}}
    \right)^{\alpha} | \mathcal{F}_{n - 1} \nobracket \right) \right|
    \nonumber\\
    & \lesssim \| D^2 \psi \|_{\infty} \quad \mathbb{E} \sum_{n =
    1}^{n_{\varepsilon}} \theta c (\varepsilon) \lesssim c_{\psi}
    (\varepsilon) T = c_{\psi, T} (\varepsilon) . \nonumber
  \end{align}
  
  Finally, since third order moments are at most of order $\theta^{3 / 2}$, we
  deduce
  \begin{eqnarray*}
    | C | & = & \left| \mathbb{E} \sum_{n = 1}^{n_{\varepsilon}} \mathbb{E} (R
    (X_{t_{n - 1}}^{\tau_{\varepsilon}}, X_{t_n}^{\tau_{\varepsilon}}) |
    \mathcal{F}_{n - 1} \nobracket) +\mathbb{E} \left( R \left( X_{t_{n -
    1}}^{\tau_{\varepsilon}}, W_{t_n}^{\tau, X_{t_{n -
    1}}^{\tau_{\varepsilon}}, t_n} \right) | \mathcal{F}_{n - 1} \nobracket
    \right) \right|\\
    & \lesssim & \| D^3 \psi \|_{\infty} \quad \mathbb{E} \sum_{n =
    1}^{n_{\varepsilon}} \mathbb{E} (| X_{t_n}^{\tau_{\varepsilon}} - X_{t_{n
    - 1}}^{\tau_{\varepsilon}} |^3 | \mathcal{F}_{n - 1} \nobracket)
    +\mathbb{E} \left( \left| W_{t_n}^{\tau, X_{t_{n -
    1}}^{\tau_{\varepsilon}}, t_n} - X_{t_{n - 1}}^{\tau_{\varepsilon}}
    \right|^3 | \mathcal{F}_{n - 1} \nobracket \right)\\
    & \lesssim & \| D^3 \psi \|_{\infty} \mathbb{E} \sum_{n =
    1}^{n_{\varepsilon}} \theta^{3 / 2} \lesssim \| D^3 \psi \|_{\infty} T
    \theta^{1 / 2} = c_{\psi, T} (\varepsilon) .
  \end{eqnarray*}
  This concludes the proof of the weak convergence.
\end{proof*}

\begin{proof*}{Proof of Lemma \ref{L: convergence discrete martingale
transforms}}
  We aim at estimating $\| f (X_T^{\tau_{\varepsilon}}) - M_T^f \|_p \assign
  (\mathbb{E} | f (X_T^{\tau_{\varepsilon}}) - M_T^f |^p)^{1 / p}$. Split
  first
  \begin{eqnarray*}
    f (X_T^{\tau_{\varepsilon}}) - f (X_0^{\tau_{\varepsilon}}) & = & f
    (B_T^{\tau_{\varepsilon}}) - f (B_0^{\tau_{\varepsilon}}) = \sum_{k =
    1}^{k_{\varepsilon}} [f (B_k^{\tau_{\varepsilon}}) - f (B_{k -
    1}^{\tau_{\varepsilon}})]\\
    & = & \sum_{k = 1}^{k_{\varepsilon}} \sum_{i = 1, 2} \partial_i f (B_{k -
    1}) \mathd B_k^i + \frac{1}{2} \sum_{k = 1}^{k_{\varepsilon}} \sum_{i, j =
    1, 2} \partial^2_{i j} f (B_{k - 1}) \mathd B^i_k \mathd B^j_k\\
    &  & + \sum_{k = 1}^{k_{\varepsilon}} R_3^f (B_{k - 1}, \mathd B_k),
  \end{eqnarray*}
  and on the other hand
  \begin{eqnarray*}
    M_T^f - M_0^f & \assign & \sum_{k = 1}^{k_{\varepsilon}} \sum_{i = 1, 2}
    \partial_i f (B_{k - 1}) \mathd B_k^i .
  \end{eqnarray*}
  Since $f (X_0^{\tau_{\varepsilon}}) = M_0^f = f (0)$, it follows simply
  \begin{eqnarray*}
    f (X_T^{\tau_{\varepsilon}}) - M_T^f & = & \frac{1}{2} \sum_{k =
    1}^{k_{\varepsilon}} \sum_{i, j = 1, 2} \partial^2_{i j} f (B_{k - 1})
    \mathd B^i_k \mathd B^j_k + \sum_{k = 1}^{k_{\varepsilon}} R_3^f (B_{k -
    1}, \mathd B_k) \backassign A + B
  \end{eqnarray*}
  For the second term above, we observe that for all $k$, all $x$,
  \[ | R_3^f (B_{k - 1} (x), \mathd B_k (x)) | \lesssim \| D^3 f \|_{\infty}
     \delta^{3 / 2}, \]
  and therefore
  \[ \| B \|_p \lesssim \sum_{k = 1}^{N^5} \| D^3 f \|_{\infty} \delta^{3 / 2}
     = \| D^3 f \|_{\infty} (N^5 \delta) \delta^{1 / 2} = c_T (\varepsilon) .
  \]
  Recalling that $\tau_{\varepsilon} = n_{\varepsilon} \delta =
  k_{\varepsilon} \theta$ or equivalently $k_{\varepsilon} = N
  n_{\varepsilon}$, we split the sum $A$ into blocks of size $N$, namely $A =
  \sum_{n = 1}^{n_{\varepsilon}} A_n$, with
  \begin{eqnarray*}
    A_n & = & \frac{1}{2} \sum_{l = 1}^N \partial^2_{11} f (B_{(n - 1) N + l -
    1})  (\mathd B_{(n - 1) N + l}^1)^2 + \partial^2_{22} f (B_{(n - 1) N + l
    - 1})  (\mathd B_{(n - 1) N + l}^2)^2\\
    & = & \delta \sum_{l = 1}^N \partial^2_{11} f (B_{(n - 1) N + l - 1})
    \tmmathbf{1} (\varepsilon_{(n - 1) N + l - 1} = - 1)\\
    &  & + \partial^2_{22} f (B_{(n - 1) N + l - 1}) \tmmathbf{1}
    (\varepsilon_{(n - 1) N + l - 1} = + 1)\\
    & = & \frac{\delta}{2} \sum_{l = 1}^N \partial^2_{11} f (B_{(n - 1) N + l
    - 1}) + \partial^2_{22} f (B_{(n - 1) N + l - 1})\\
    &  & + \frac{\delta}{2} \sum_{l = 1}^N [\partial^2_{22} f (B_{(n - 1) N +
    l - 1}) - \partial^2_{11} f (B_{(n - 1) N + l - 1})] \varepsilon_{(n - 1)
    N + l - 1}\\
    & = & \delta \sum_{l = 1}^N \partial^2_{22} f (B_{(n - 1) N + l - 1})
    \varepsilon_{(n - 1) N + l - 1},
  \end{eqnarray*}
  where we used $\tmmathbf{1} (\varepsilon = \pm 1) = \frac{1}{2} + \left(
  \tmmathbf{1} (\varepsilon = \pm 1) - \frac{1}{2} \right) = \frac{1}{2} \pm
  \frac{\varepsilon}{2}$ and the harmonicity of $f$. We split further
  
  \begin{align}
    A_n & = \delta \sum_{l = 1}^N [\partial^2_{22} f (B_{(n - 1) N + l - 1}) -
    \partial^2_{22} f (B_{(n - 1) N})] \varepsilon_{(n - 1) N + l - 1}
    \nonumber\\
    & \hspace{8em} + \delta \partial^2_{22} f (B_{(n - 1) N}) \sum_{l = 1}^N
    \varepsilon_{(n - 1) N + l - 1} \nonumber\\
    & \backassign B_n + C_n \nonumber
  \end{align}
  
  For $B_n$, we observe $| \partial^2_{22} f (B_{(n - 1) N + l - 1}) -
  \partial^2_{22} f (B_{(n - 1) N}) | \lesssim \| D^3 f \|_{\infty} l
  \sqrt{\delta}$, therefore $\| B_n \|_p \lesssim N^2 \delta^{3 / 2}$ and
  \[ \left\| \sum_{n = 1}^{N^4} B_n \right\|_p \lesssim N^6 \delta^{3 / 2} =
     (N^5 \delta) N (T N^{- 5})^{1 / 2} = c_{\varepsilon} (T) . \]
  Next the norm of $C_n$ is estimated
  \[ \| C_n \|_p \lesssim \delta \| D^2 f \|_{\infty}  \left( \mathbb{E}
     \left| \sum_{l = 1}^N \varepsilon_{(n - 1) N + l - 1} \right|^p
     \right)^{1 / p} . \]
  Notice that the sum above is $\sum_{l = 1}^N \varepsilon_{(n - 1) N + l - 1}
  = \mathd X_n^1 + \mathd X^2_n$, and we know from the moment estimates that
  $\| \mathd X_n^i \|_p \lesssim \theta^{1 / 2}$, $i = 1, 2$. We conclude
  \[ \left\| \sum_{n = 1}^{N^4} C_n \right\|_p \lesssim N^4 \delta \| D^2 f
     \|_{\infty} (N \delta)^{1 / 2} = (N^5 \delta) \| D^2 f \|_{\infty} N^{- 1
     / 2} \delta^{1 / 2} = c_T (\varepsilon) \]
  This concludes the proof of Lemma \ref{L: convergence discrete martingale
  transforms}.
\end{proof*}

\section{Proofs of the main results}\label{S: main results}

\begin{proof*}{Proof of Theorem \ref{T: convergence}}
  {\dueto{Convergence of $L^p$ norms of martingales}}Recall that we want to
  prove
  \[ \lim_{T \rightarrow \infty} \lim_{\varepsilon \rightarrow 0} \mathbb{E} |
     M_T^f |^p =\mathbb{E} | f (W^{\tau}_{\infty}) |^p, \]
  or equivalently in terms of norms
  \[ \lim_{T \rightarrow \infty} \lim_{\varepsilon \rightarrow 0} \| M_T^f
     \|_p = \| f (W^{\tau}_{\infty}) \|_p . \]
  Split first as the sum of three differences
  \[ \begin{array}{l}
       \| f (W^{\tau}_{\infty}) \|_p - \| M_T^f \|_p\\
       \qquad = \| f (W^{\tau}_{\infty}) \|_p - \| f (W^{\tau}_T) \|_p + \| f
       (W^{\tau}_T) \|_p - \| f (X^{\tau_{\varepsilon}}_T) \|_p + \| f
       (X^{\tau_{\varepsilon}}_T) \|_p - \| M_T^f \|_p\\
       \qquad \backassign A + B + C
     \end{array} \]
  As seen before, we have $\mathbb{P} (\tau > T) = c (T)$, therefore $| A | =
  c (T)$. For the third term, we have simply $| C | \leqslant \| f
  (W^{\tau}_T) - f (X^{\tau_{\varepsilon}}_T) \|_p = c_T (\varepsilon)$ thanks
  to Lemma \ref{L: convergence discrete martingale transforms}.
  
  For the second term, define successively on the boundary $\partial
  \mathbb{D}$, $\psi (x) \assign | f (x) |_X^p$ and $\psi_{\eta} \assign \psi
  \ast \rho_{\eta}$ a mollified version of $\psi$ tending to $\psi$ when
  $\eta$ goes to zero. We also denote $\psi$ (resp. $\psi_{\eta}$) defined on
  $\mathbb{D}$ the Poisson extension of $\psi_{| \partial \mathbb{D}}$ (resp.
  $\psi_{| \partial \mathbb{D}}$). Since $f \in L^{\infty} (\partial
  \mathbb{D}; X)$, it follows that $\psi$ and $\psi_{\eta}$ are bounded in
  $\mathbb{D}$, and $\psi_{\eta} = \psi + c (\eta)$ in $L^{\infty}
  (\mathbb{D})$. Finally, notice that if $| x | < 1$, then we have for the
  Poisson extensions $\psi (x) \neq | f (x) |_X^p$. However $\psi (x) = | f
  (x) |_X^p + c (\varepsilon)$ for those $x$'s next to the boundary, i.e. $1 -
  \varepsilon < | x | \leqslant 1$. We can now estimate the first term of $B$:
  \begin{eqnarray*}
    \| f (W_T^{\tau}) \|_p & = & \mathbb{E} (| f (W_T^{\tau}) |^p)^{1 / p}
    =\mathbb{E} (| f (W_T^{\tau}) |^p  | T > \tau)^{1 / p} + c (T)\\
    & = & \mathbb{E} (\psi (W_T^{\tau})  | T > \tau)^{1 / p} + c (T)
    =\mathbb{E} (\psi (W_T^{\tau}))^{1 / p} + c (T)\\
    & = & \mathbb{E} (\psi_{\eta} (W_T^{\tau}))^{1 / p} + c (\eta) + c (T)
  \end{eqnarray*}
  where we have used that $W_T^{\tau} \in \partial \mathbb{D}$ when $T >
  \tau$. Similarly, since $\mathbb{P} (\tau_{\varepsilon} > T) = c (T)$ and $1
  - \varepsilon \leqslant X_T^{\tau_{\varepsilon}} \leqslant 1$ for $T >
  \tau_{\varepsilon}$, we have
  \begin{eqnarray*}
    \| f (X_T^{\tau_{\varepsilon}}) \|_p & = & \mathbb{E} (| f
    (X_T^{\tau_{\varepsilon}}) |^p)^{1 / p} =\mathbb{E} (| f
    (X_T^{\tau_{\varepsilon}}) |^p | T > \tau_{\varepsilon})^{1 / p} + c (T)\\
    & = & \mathbb{E} (\psi (X_T^{\varepsilon}) | T > \tau_{\varepsilon})^{1 /
    p} + c (\varepsilon) + c (T)\\
    & = & \mathbb{E} (\psi (X_T^{\varepsilon}))^{1 / p} + c (\varepsilon) + c
    (T)\\
    & = & \mathbb{E} (\psi_{\eta} (X_T^{\varepsilon}))^{1 / p} + c (\eta) + c
    (\varepsilon) + c (T)
  \end{eqnarray*}
  It follows,
  \begin{eqnarray*}
    | B | & \leqslant & | \mathbb{E} (\psi_{\eta} (W_T^{\tau}))^{1 / p}
    -\mathbb{E} (\psi_{\eta} (X_T^{\varepsilon}))^{1 / p} | + c (\eta) + c
    (\varepsilon) + c (T)\\
    & = & c_{\eta, T} (\varepsilon) + c (\eta) + c (\varepsilon) + c (T)
  \end{eqnarray*}
  where we used Lemma \ref{L: weak convergence} for the second line. Finally
  \begin{eqnarray*}
    \| f (W^{\tau}_{\infty}) \|_p - \| M_T^f \|_p & = & A + B + C = c_{\eta,
    T} (\varepsilon) + c (\eta) + c (\varepsilon) + c (T) .
  \end{eqnarray*}
  Fix any small $\eta > 0$, choose $T > 0$ large enough so that $c (T)
  \leqslant \eta$, then $\varepsilon > 0$ small enough so that $c_{\eta, T}
  (\varepsilon) + c (\varepsilon) \leqslant \eta$. Hence
  \[ \lim_{T \rightarrow \infty} \lim_{\varepsilon \rightarrow 0} | \| f
     (W^{\tau}_{\infty}) \|_p - \| M_T^f \|_p | \leqslant c (\eta), \]
  therefore
  \[ \lim_{T \rightarrow \infty} \lim_{\varepsilon \rightarrow 0} \| M_T^f
     \|_p = \| f (W^{\tau}_{\infty}) \|_p \]
  as desired. This concludes the proof of Theorem \ref{T: convergence}.
\end{proof*}

\begin{proof*}{Proof of Theorem \ref{T: H less than S}}
  This is now a direct consequence of Theorem \ref{T: convergence}. Let $f \in
  L^p (\partial \mathbb{D})$. Its $L^p$ norm is directly related to the
  stochastic $L^p$ norm
  \[ \| f (W_{\infty}^{\tau}) \|_p \assign (\mathbb{E} | f (W_{\infty}^{\tau})
     |_X^p)^{1 / p} = \left( \int_{\partial \mathbb{D}} | f (z) |_X^p 
     \frac{\mathd z}{2 \pi} \right)^{1 / p} = \frac{1}{(2 \pi)^{1 / p}}  \| f
     \|_{L^p (\partial \mathbb{D})}, \]
  and the same relation holds for the smooth function $g \assign \mathcal{H}
  f$. From Theorem \ref{T: convergence} we know
  \[ \| f (W_{\infty}^{\tau}) \|_p = \lim_{T \rightarrow \infty}
     \lim_{\varepsilon \rightarrow 0} \| M_T^f \|_p, \qquad \| g
     (W_{\infty}^{\tau}) \|_p = \lim_{T \rightarrow \infty} \lim_{\varepsilon
     \rightarrow 0} \| M_T^g \|_p, \]
  and from Lemma \ref{L: Lp estimate for Mng} we know that for all $T > 0$,
  $\varepsilon > 0$,
  \[ \| M_T^g \|_p \leqslant \| \mathcal{S} \|_{p \rightarrow p}  \| M_T^f
     \|_p . \]
  It follows that for any $f \in L^p (\partial \mathbb{D})$,
  \[ \| g \|_{L^p (\partial \mathbb{D})} = \| \mathcal{H} f \|_{L^p (\partial
     \mathbb{D})} \leqslant \| \mathcal{S} \|_{p \rightarrow p} \| f \|_{L^p
     (\partial \mathbb{D})}, \]
  that is $\| \mathcal{H} \|_{p \rightarrow p} \leqslant \| \mathcal{S} \|_{p
  \rightarrow p}$.
\end{proof*}

\section{Averaging of the dyadic Hilbert transform}\label{S: averaging dyadic
Hilbert transform}

We prove that the average of the dyadic Hilbert transform $\mathcal{S}$ in the
sense of {\cite{Pet2000b}} is null. For that, let $\mathcal{D}^{\alpha, r} =
\{ 2^r I + \alpha : I \in \mathcal{D} \}$ the dilated and translated dyadic
grid. Here, let $1 \leqslant r < 2$ and $\alpha \in \mathbb{R}$. Denote by
$h^{\alpha, r}_I$ the corresponding $L^2$-normalized Haar functions and by
$\mathcal{S}^{\alpha, r}$ the dyadic Hilbert transform associated to
$\mathcal{D}^{\alpha, r}$. Since in the usual sense,
\[ \mathcal{S}^{\alpha, r} : f (x) \mapsto \sum_{I \in \mathcal{D}^{\alpha,
   r}} [- \langle f, h^{\alpha, r}_{I_-} \rangle h^{\alpha, r}_{I_+} (x) +
   \langle f, h^{\alpha, r}_{I_+} \rangle h^{\alpha, r}_{I_-} (x)], \]
the kernel of $\mathcal{S}^{\alpha, r}$ is
\[ K^{\alpha, r} (t, x) = \sum_{I \in \mathcal{D}^{\alpha, r}} [- h^{\alpha,
   r}_{I_-} (t) h^{\alpha, r}_{I_+} (x) + h^{\alpha, r}_{I_+} (t) h^{\alpha,
   r}_{I_-} (x)] . \]
Here is the illustration of the sign distribution of a single term for a given
interval $I$.

\begin{center}
\includegraphics[width=5.86323051948052cm,height=4.85888921684376cm]{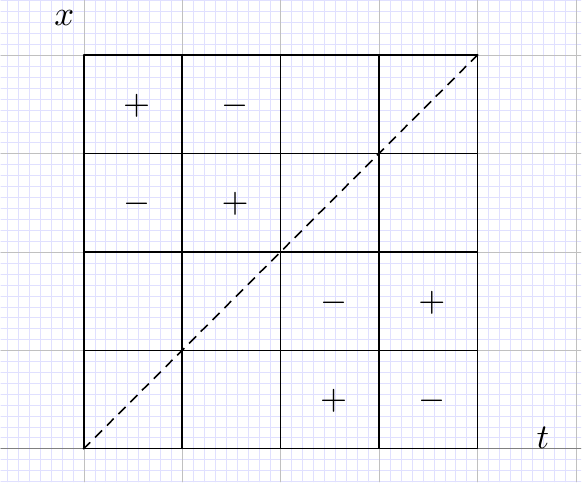}
\end{center}

Following the strategy of the second author in {\cite{Pet2000b}}, the average
of the kernel by dilation and translation we consider is:
\[ \mathbb{E}_r \mathbb{E}_{\alpha} K^{\alpha, r} (t, x) = \frac{1}{\log 2}
   \int^2_1 \lim_{R \rightarrow \infty} \frac{1}{2 R} \int^R_{- R} K^{\alpha,
   r} (t, x) \mathd \alpha \frac{\mathd r}{r} . \]
\begin{lemma}
  We have $\mathbb{E}_r \mathbb{E}_{\alpha} K^{\alpha, r} (t, x) = 0.$
\end{lemma}

\begin{proof}
  While these averages can be calculated explicitly, one can also argue as
  follows. The kernel $K^{\alpha, r} (t, x)$ can be split into a sum of
  partial kernels
  \[ K_-^{\alpha, r} (t, x) = \sum_{I \in \mathcal{D}^{\alpha, r}} h^{\alpha,
     r}_{I_+} (t) h^{\alpha, r}_{I_-} (x) \]
  and
  \[ K_+^{\alpha, r} (t, x) = \sum_{I \in \mathcal{D}^{\alpha, r}} -
     h^{\alpha, r}_{I_-} (t) h^{\alpha, r}_{I_+} (x) . \]
  $K_-^{\alpha, r} (t, x)$ is supported under the diagonal, $t > x$ while
  $K_+^{\alpha, r} (t, x)$ is supported over the diagonal, $t < x$. Clearly
  the averages produce no mass on the diagonal. Observe now that the average
  of the even kernel $K_-^{\alpha, r} (t, x) - K_+^{\alpha, r} (t, x)$ is
  dilation invariant, translation invariant and symmetric. It therefore
  represents the zero operator. We may then conclude that the averages of the
  partial kernels themselves are 0. The claim follows.
\end{proof}

\

\

\

\

\end{document}